\documentclass[12pt]{article}
\usepackage[T1]{fontenc}
\usepackage[utf8]{inputenc}
\usepackage[english]{babel}

\usepackage{amsmath}
\usepackage{commath}
\usepackage{enumitem}
\usepackage{amssymb}
\usepackage{amsthm}
\usepackage{xcolor} 
\usepackage{comment}

\usepackage[a4paper,top=2.1cm,bottom=2.60cm,left=2.5cm,right=2.5cm]{geometry}

\usepackage{todonotes}

\theoremstyle{plain}
\newtheorem{theorem}{Theorem}[section]
\newtheorem{lem}[theorem]{Lemma}
 
\newtheorem{corol}[theorem]{Corollary}

\theoremstyle{definition} 
\newtheorem{rem}[theorem]{Remark}
\newtheorem{defn}[theorem]{Definition}

\newcommand{\xn}{X_{\rm \scriptscriptstyle N}}

\begin{document}

\author{ Pier Domenico Lamberti \thanks{Dipartimento di Tecnica e Gestione dei Sistemi Industriali, Universit\`a degli Studi di Padova, Stradella S. Nicola 3, 36100 Vicenza, Italy.},
 Paolo Luzzini \thanks{Dipartimento di Matematica `Tullio Levi-Civita', Universit\`a degli Studi di Padova, Via Trieste 63, 35121 Padova, Italy.}, 
 Michele Zaccaron \thanks{Institut Fresnel, 
Faculte des Sciences de Saint J\'er\^ome,
Avenue Escadrille Normandie-Niemen,
13397 Marseille, 
France}}

\title{Permittivity optimization for Maxwell's eigenvalues}
\maketitle

\abstract{We formulate an optimization problem for the dependence of  the eigenvalues of Maxwell's equations in a cavity upon variation of the electric permittivity and we prove a corresponding Maximum Principle.}  

\vspace{.4cm}

\vspace{.4cm}


\section{Introduction}
 
Let $\Omega$ be a bounded Lipschitz domain of $\mathbb{R}^3$ and $\varepsilon$ be a $3 \times 3 $ real symmetric matrix-valued  function.  
We consider the electric eigenvalue problem for the $\mathit{curl}\mathit{curl}$ operator in $\Omega$ with permittivity 
$\varepsilon$, that is:
\begin{equation}\label{prob:eig}
\begin{cases}
\operatorname{curl}\operatorname{curl} E = \lambda \, \varepsilon E \qquad &\mbox{ in } \Omega,\\
\mathrm{div}\, \varepsilon E = 0 & \mbox{ in } \Omega,\\
\nu \times E = 0 \qquad &\mbox{ on } \partial \Omega,
\end{cases}
\end{equation}
where $\nu$ denotes the outer unit normal to the boundary of $\Omega$. 

The above problem is strictly related to Maxwell's equations. Indeed, if $\Omega$ represents 
an electromagnetic cavity, the time-harmonic Maxwell's equations are:
\begin{equation}\label{ME}
\operatorname{curl} E - i\omega \mu H = 0, \qquad \operatorname{curl} H + i\omega \varepsilon E = 0 \quad \mbox{ in } \Omega.
\end{equation}
The vector fields $E$ and $H$ are respectively the electric and magnetic parts of the electromagnetic field, while $\varepsilon$ and $\mu$ are the electric conductivity and magnetic permeability of the medium filling $\Omega$, and $\omega$ is the frequency. 
If one assumes that the magnetic permeability is constant and equal to one, by applying the $\operatorname{curl}$ operator to the first equation in \eqref{ME} and then using the second equation one obtains:
\[
\operatorname{curl}\operatorname{curl} E = \omega^2 \, \varepsilon E.
\]
Hence one recovers the first equation in \eqref{prob:eig}  by setting $\lambda = \omega^2$. 
The second equation in \eqref{prob:eig} derives, at least in the case $\omega$ is not zero, directly  from the fact that the vector field $\varepsilon E$ is a curl, as one can see from the second equation in \eqref{ME}, thus it is divergence-free. 
Finally, if the medium is a perfect conductor then the electric field must satisfy the boundary condition $\nu \times E = 0$, which is the third equation in \eqref{prob:eig}. For an exhaustive  introduction to the mathematical theory of Maxwell's equations we refer e.g. to the monograph \cite{kihe} by   Kirsch and Hettlich. 

We remark that, in addition to \cite{kihe}, many other authors investigated Maxwell's equations due to their importance both from the theoretical and applied point of view. We mention for example some established monographs  \cite{ces, DaLi90, GiRa86, monk,RoStYa12} as well as the recent papers \cite{ammari, bauer, fermar, CaLaMo17, LaSt20, LaZa23, Pa17, Yi12}.

Coming back to problem \eqref{prob:eig}, under suitable assumptions on $\Omega$ and $\varepsilon$, it is well-known that its spectrum  is discrete  and consists of   a divergent sequence of  $\varepsilon$-dependent, non-negative eigenvalues $\{\lambda_j[\varepsilon]\}_{j \in \mathbb{N}}$ of finite multiplicity (see e.g. \cite[Thm. 4.34]{kihe}):
\[
0\leq \lambda_1[\varepsilon] \leq \lambda_2[\varepsilon] \leq \cdots \leq \lambda_n[\varepsilon] \leq \cdots \nearrow +\infty,
\]
where we repeat each eigenvalue in accordance with its multiplicity.

In the present paper we are interested in the dependence of  $\lambda_n[\varepsilon]$ upon variation of $\varepsilon$ with reference  to an  optimization problem that we formulate as follows.  

We begin with observing that  problem \eqref{prob:eig} can be considered as the vectorial version of the scalar  eigenvalue problem for the Laplace operator 
\begin{equation}\label{massden}
\left\{
\begin{array}{ll}
-\Delta u = \mu\rho u \quad &\mbox{in}\  \Omega,\vspace{2pt} \\
u=0& \mbox{on}\  \partial\Omega,
\end{array}
\right.
\end{equation}
that  models a vibrating membrane with shape $\Omega$ and mass density  $\rho $.  Note that the total mass of  the membrane is given by $\int_{\Omega}\rho dx$.    The eigenvalues of \eqref{massden} are represented  by a diverging sequence 
$0< \mu_1[\rho] \leq \mu_2[\rho] \leq \cdots \leq \mu_n[\rho] \leq \cdots $ depending on $\rho$.   A classical problem in optimization theory is to study the extremum problems
\begin{equation}\label{extremummass}
\min_{\int_{\Omega}\rho \, dx={\rm const}}\mu_n[\rho],\ \ {\rm and}\ \   \max_{\int_{\Omega}\rho \, dx={\rm const}}\mu_n[\rho]
\end{equation}
where one looks for minima and maxima of the eigenvalues of \eqref{massden} upon variation of $\rho$ subject to the constraint that  the total mass  is fixed.  We refer to   Cox \cite{Co94}, Cox and McLaughlin \cite{CoMc90A, CoMc90B}, Krein \cite{Kr55}, Lamberti \cite{La09}  and  Lamberti and Provenzano \cite{LaPr13}  for several results regarding problem \eqref{extremummass} and its variants.

As far as problem  \eqref{prob:eig} is concerned, we note that the energy $\mathbf{En}(E)$  of 
the electric field $E$  is given by 
 $\mathbf{En}(E)=\frac{1}{2}\int_\Omega\varepsilon E \cdot E \,dx$.  It is obvious that
 $$
\mathbf{En}(E) \le \frac{\sup^2_{x\in\Omega} |E(x)| }{2}\int_{\Omega} \varepsilon \frac{E}{|E|} \cdot \frac{E}{|E|}  dx  \le \frac{\sup^2_{x\in\Omega} |E(x)| }{2}\int_{\Omega}| \varepsilon |_{\mathcal{F}}dx 
$$
where 
$
\frac{1}{2}\int_{\Omega}| \varepsilon |_{\mathcal{F}}dx 
$
can be seen as  an upper bound for the energy of the normalised electric field $E/|E|$.  Here $|\varepsilon |_{\mathcal{F}}=(\sum_{i,j=1}^3\varepsilon_{i,j}^2)^{1/2}$ is the Frobenius norm of the matrix $\varepsilon$.  Thus, from this point of view, it appears that  a   natural vectorial  counterpart of problem \eqref{extremummass} could be the
problem 
\begin{equation}\label{extremumenergy}
\min_{\int_{\Omega}|\varepsilon|_{\mathcal{F}}  dx={\rm const}}\lambda_n[\varepsilon],\ \ {\rm and}\ \   \max_{\int_{\Omega}|\varepsilon|_{\mathcal{F}} dx={\rm const}}\lambda_n[\varepsilon].
\end{equation}

The main aim of the present paper is to further motivate problem  \eqref{extremumenergy} by showing that  a number of results  related to 
\eqref{extremummass} can be extended in a natural way to \eqref{extremumenergy}.  In particular, we follow the approach of 
\cite{La09}, \cite{LaPr13}  and we prove a `Maximum Principle' for simple eigenvalues and the elementary symmetric functions  of multiple eigenvalues 
$\lambda_n[\varepsilon]$.   Namely, problem \eqref{extremumenergy} is recasted in  the more natural form

\begin{equation}\label{extremumenergysym}
\min_{\int_{\Omega}|\varepsilon|  dx={\rm const}}\Lambda_{F,s}[\varepsilon],\ \ {\rm and}\ \   \max_{\int_{\Omega}|\varepsilon |dx={\rm const}}\Lambda_{F,s}[\varepsilon]
\end{equation}
where 
\[
\Lambda_{F,s}[\varepsilon] := \sum_{\substack{j_1,\dots,j_s \in F \\ j_1<\dots<j_s}} \lambda_{j_1}[\varepsilon] \cdots \lambda_{j_s}[\varepsilon]
\]
are the elementary symmetric functions of the  the eigenvalues with index in a fixed finite set $F$ and $s=1,\dots , |F|$.

 In order to prove our Maximum Principle, first we prove that any eigenvalue $\lambda_n[\varepsilon]$ is weakly* continuous with respect to $\varepsilon\in W^{1,\infty}(\Omega)^{3 \times 3}$,  then we prove that problem \eqref{extremumenergysym} has no local extrema.  Then we deduce that, for any weakly* compact subset $C$ of $W^{1,\infty}{(\Omega)^{3 \times 3}}$ of admissible permittivities, problem \eqref{extremumenergysym}   is solvable for $\varepsilon \in C$ and the corresponding points   of maxima and minima  cannot be in the interior of $C$. 

Incidentally, we point out  the reason why we consider the symmetric functions of multiple  eigenvalues rather than the eigenvalues themselves. This  is motivated  by the fact that the former have turned out to be natural objects in spectral optimization for elliptic problems, in particular because  they depend real-analytically on $\varepsilon$ (as proved in \cite{LuZa22})  and the formulas for their differentials are invariant with respect to the choice of an orthonormal basis of eigenvectors 
(see \cite{Lala06,Za23} for the case of the optimization of domain dependent problems). 
Note that the eigenvalues themselves are not even differentiable with  respect to $\varepsilon$  unless they are simple.

This paper is organized as follows. Section~\ref{sec:pre} is devoted to preliminaries and notation,  Section~\ref{sec:weak} to the weak* continuity of the eigenvalues, Section~\ref{sec:permop} to the optimization problem and the corresponding Maximum Principle. In the appendix we give a detailed proof of the Auchmuty's Principle for Maxwell's eigenvalues used in Section \ref{sec:weak}.

\section{Some preliminaries}\label{sec:pre}

Since we deal with self-adjoint operators, we consider only spaces of real-valued functions. In this sense, we denote by $L^2(\Omega)^3= L^2(\Omega;\mathbb{R}^3)$ the space of real square integrable vector fields, endowed with the natural inner product 
\[
\int_{\Omega} u \cdot v \,dx  = \int_{\Omega} (u_1v_1+u_2v_2+u_3v_3) \,dx
\]
for all $u = (u_1,u_2,u_3)$, $v= (v_1,v_2,v_3)$ in $ L^2(\Omega)^3$.
We also introduce the spaces  $L^\infty(\Omega)^{3 \times 3}$ and  $W^{1,\infty}(\Omega)^{3 \times 3}$ of real matrix-valued functions $M=\left( M_{ij} \right)_{1 \leq i,j \leq 3}:\Omega \to   \mathbb{R}^{3 \times 3}$ 
whose components are  in $L^\infty(\Omega)$ and $W^{1,\infty}(\Omega)$, respectively.
We endow these spaces   with the following norms
\begin{equation*}
\norm{M}_{L^\infty(\Omega)^{3 \times 3}} := \max_{1 \leq i,j \leq 3} \norm{M_{ij}}_{L^\infty(\Omega)}
\end{equation*}
and 
\begin{equation*}
\norm{M}_{W^{1,\infty}(\Omega)^{3 \times 3}} := \max_{1 \leq i,j \leq 3} \norm{M_{ij}}_{W^{1,\infty}(\Omega)}.
\end{equation*}
Letting $M =\left( M_{ij} \right)_{1 \leq i,j \leq 3}\in L^\infty(\Omega)^{3 \times 3}$, the following inequalities are easy to verify:
\begin{equation*}
\abs{M \xi} \leq 3 \norm{M}_{L^\infty(\Omega)^{3 \times 3}} |\xi|, \qquad \abs{M \xi \cdot \xi} \leq 3 \norm{M}_{L^\infty(\Omega)^{3 \times 3}} \abs{\xi}^2,
\end{equation*}
for all $\xi \in \mathbb{R}^3$ and a.e. in $\Omega$.   They will be used throughout the paper.

Before proceeding any further, we need to specify the assumptions on the parameters involved in problem \eqref{prob:eig}.
The following set defines the class of admissible permittivities $\varepsilon$ under consideration:
\begin{align*}
\mathcal{E}:= \Big\{\varepsilon \in\, & W^{1,\infty} \left(\Omega\right)^{3 \times 3} \cap  \mathrm{Sym}_3 (\Omega) : \\
&\exists \, c>0 \text{ s.t. }  \varepsilon(x) \, \xi \cdot \xi \geq c \, \abs{\xi}^2 \text{ for a.a. } x \in  \Omega, \text{ for all }\xi \in \mathbb{R}^3\Big\},
\end{align*}
where $\mathrm{Sym}_3 (\Omega)$ denotes the set of $(3 \times 3)$-symmetric matrix valued functions in $\Omega$.
We regard $W^{1,\infty} \left(\Omega\right)^{3 \times 3} \cap  \mathrm{Sym}_3 (\Omega)$ as a Banach subspace of  $W^{1,\infty} \left(\Omega\right)^{3 \times 3}$.
Observe that the set $\mathcal{E}$ is  open in  $W^{1,\infty}(\Omega)^{3 \times 3}\cap  \,\, \mathrm{Sym}_3 (\Omega)$ (see \cite[\S 2]{LuZa22}). 
In order to prove a  weak* continuity result (cf. Theorem \ref{weak*:cont:thm}), it will be necessary to introduce some type of compactness. To do this, we will further restrict the class of  permittivities to have some boundedness of the norms. Therefore, we introduce the following set. Let $\alpha,\beta,\gamma>0$   such that $0 < \alpha < \beta$. Then
\begin{equation} \label{def:Aabc}
\begin{split}
\mathcal{A}_{\alpha,\beta,\gamma} := \bigg\{\varepsilon & \in W^{1,\infty}(\Omega)^{3\times 3} \cap \mathrm{Sym}_3(\mathbb{R}) : \alpha |\xi|^2  \leq \varepsilon (x) \xi \cdot \xi \leq \beta |\xi|^2 \text{ for a.e. }x \in \Omega, \\
& \text{ for all } \xi \in \mathbb{R}^3, \ \norm{\nabla \varepsilon_{ij}}_{L^\infty(\Omega)^3} \leq \gamma \text{ for all } 1\leq i,j \leq 3\bigg\}.
\end{split}
\end{equation}
\begin{rem} 
It is immediate to realize that $\mathcal{A}_{\alpha,\beta,\gamma} \subset \mathcal{E}$. Also note that $\mathcal{A}_{\alpha,\beta,\gamma}$ is   a bounded subset of $L^{\infty}(\Omega)^{3\times 3}$, and therefore also of  $W^{1,\infty}(\Omega)^{3\times 3}$. Indeed one can prove the boundedness of the diagonal terms of $\varepsilon$ by taking as $\xi$ in the definition \eqref{def:Aabc} the canonical basis $\{e_i\}_{i=1,2,3}$ of $\mathbb{R}^3$. Then, the boundedness of the off-diagonal terms follows by using $\xi = e_i+e_j$, $i\neq j$, and  the symmetry of $\varepsilon$.
\end{rem}
As for  the set $\Omega$, for the moment we  just consider it a bounded Lipschitz domain of $\mathbb{R}^3$. Further down, after we will introduce the Gaffney inequality (cf. \eqref{ineq:GF}), we will add more requirements on its regularity.

Next we introduce the natural functional framework to set  problem \eqref{prob:eig}.
Let $\varepsilon \in \mathcal{E}$. We denote by 
$L^2_\varepsilon(\Omega)$ 
 the space $L^2(\Omega)^3$ endowed with the inner product
\begin{equation*}
\langle u, v \rangle_\varepsilon := \int_\Omega \varepsilon u \cdot v \, dx \qquad \forall u,v \in L^2(\Omega)^3.
\end{equation*}
It is easy to see that the above inner product induces a norm equivalent to the standard $L^2$-norm.

The space of square integrable vector fields with square integrable (weak) curl is denoted by
\begin{equation*}
H(\mathrm{curl}, \Omega) = \{ u \in L^2(\Omega)^3 : \operatorname{curl}u \in L^2(\Omega)^3\}.
\end{equation*}
Endowed with  the natural inner product
\begin{equation*}
\langle u,v\rangle_{H(\operatorname{curl}, \Omega)} := \int_\Omega \varepsilon u \cdot v \, dx + \int_\Omega \operatorname{curl}u \cdot \operatorname{curl}v \, dx \qquad \forall\, u,v \in H(\operatorname{curl},\Omega)
\end{equation*}
it becomes a Hilbert space.
Recall that $u\in L^2(\Omega)^3$ has a weak (or distributional) curl if  there exists a vector field $\operatorname{curl}u \in  L^2(\Omega)^3$ such that 
\begin{equation} \label{defweakcurl}
\int_\Omega u \cdot \operatorname{curl}\varphi \, dx = \int_\Omega \operatorname{curl}u \cdot \varphi \, dx \quad \forall \varphi \in C^\infty_c(\Omega)^3.
\end{equation}
By $H_0(\mathrm{curl}, \Omega)$ we denote  the closure of smooth compactly supported vector fields   in $\Omega$ in the norm  of $H(\mathrm{curl}, \Omega)$. It is precisely the space composed of those vector fields $u \in H(\operatorname{curl},\Omega)$ for which \eqref{defweakcurl} holds for any $\varphi \in C^\infty(\Omega)^3$.
Moreover, it can be characterized in the following way (cf. \cite[Thm. 2.12]{GiRa86}): 
\begin{equation*}
H_0(\operatorname{curl},\Omega) = \set{u \in H(\operatorname{curl},\Omega) :  \nu \times u\rvert_{\partial \Omega}=0},
\end{equation*}
that is, it is the space of vector fields in $H(\operatorname{curl},\Omega)$ whose \emph{tangential trace} is zero.
In the case $u$ is regular enough, its tangential trace is precisely the cross product between the  outer unit normal $\nu$ and the restriction of $u$ to $\partial \Omega$, while in general it is just an element of the dual space $H^{-1/2}(\partial \Omega)^3$ (denoted by $\gamma_\tau u$). For the sake of simplicity, we will always use the former notation also when referring to the general case. 
For more details we refer to  \cite[Ch. 2]{GiRa86} or \cite[Ch. IX-A \S 1.2]{DaLi90}.
Observe that the space $H_0(\operatorname{curl},\Omega)$ naturally encodes the boundary condition of problem \eqref{prob:eig}.

Analogously, we define the following Hilbert space:
\begin{equation*}
H(\operatorname{div} \varepsilon, \Omega)=\{ u  \in L^2(\Omega)^3 : \operatorname{div}\varepsilon u \in L^2(\Omega)\}
\end{equation*} 
endowed  with the  inner product
\begin{equation*}
\langle u,v\rangle_{H(\operatorname{div} \varepsilon, \Omega)} := \int_\Omega \varepsilon u \cdot v \, dx + \int_\Omega \operatorname{div}(\varepsilon u)  \operatorname{div}(\varepsilon v) \, dx \qquad  \forall\, u,v \in H(\operatorname{div}\varepsilon,\Omega).
\end{equation*}
We say that $u \in L^2(\Omega)^3$ has a weak (or distributional) $\varepsilon$-divergence if there exists a function $\operatorname{div}(\varepsilon u) \in L^2(\Omega)$ such that 
\begin{equation*}
\int_\Omega \varepsilon u \cdot \nabla \varphi \, dx = - \int_\Omega \operatorname{div}(\varepsilon u) \, \varphi \, dx \quad \forall \varphi \in C^\infty_c(\Omega).
\end{equation*}
If a vector field $u$ belongs to  $H^1(\Omega)^3$ and $\varepsilon \in \mathcal{E}$ then by standard rules of calculus, one has that
\begin{equation} \label{div:matrixxvector}
\operatorname{div}(\varepsilon u) = \operatorname{tr}(\varepsilon Du) + \operatorname{div}\varepsilon \cdot u
\end{equation}
a.e. in $\Omega$, where $Du$ denotes the Jacobian of $u$. Here, if $M$ is a matrix then $\operatorname{tr}(M)$ denotes the trace of a matrix $M$. Moreover, if $M \in W^{1,\infty}(\Omega)^{  3\times 3}$ then $\operatorname{div}M$ is the vector field defined by
\begin{equation*}
\operatorname{div}M := \left( \operatorname{div}M^{(1)}, \operatorname{div}M^{(2)}, \operatorname{div}M^{(3)} \right)
\end{equation*}
with $M^{(k)}$ denoting the $k$-th column of  $M$.

Furthermore, we introduce the space
\[
\xn^\varepsilon(\Omega) := H_0(\mathrm{curl}, \Omega) \cap H(\mathrm{div}\,\varepsilon, \Omega)
\]
equipped with inner product
\begin{equation*}
\langle u,v\rangle_{\xn^\varepsilon(\Omega)} := \int_\Omega \varepsilon u \cdot v \, dx  + \int_\Omega \operatorname{curl}u \cdot \operatorname{curl}v
 \, dx + \int_\Omega \operatorname{div} (\varepsilon u) \, \operatorname{div}(\varepsilon v) \, dx
\end{equation*}
for all $u,v \in \xn^\varepsilon(\Omega)$.
Its norm is 
\begin{equation*}
\|u\|^2_{\xn^\varepsilon (\Omega)}:=\langle \varepsilon u, u\rangle_{L^2(\Omega)^3}+ \norm{\operatorname{curl}u}^2_{L^2(\Omega)^3} + \|\operatorname{div}\varepsilon u\|^2_{L^2(\Omega)}.
\end{equation*}
Finally, we set 
\begin{equation*}
\begin{split}
\xn^\varepsilon(\mathrm{div}\,\varepsilon\, 0,\Omega) &:= \set{u \in \xn^\varepsilon(\Omega) : \mathrm{div} \, (\varepsilon u) = 0} \\
&= \set{u \in L^2(\Omega)^3 : \operatorname{curl}u \in L^2(\Omega)^3, \operatorname{div}(\varepsilon u)=0, \nu \times u\rvert_{\partial \Omega}=0}.
\end{split}
\end{equation*}

Under suitable assumptions on $\Omega$, we have the validity of the \emph{Gaffney inequality} (or Gaffney-Friedrichs inequality), which reads as follows: there exists a constant $C>0$ (in general depending on $\varepsilon \in \mathcal{E}$) such that
\begin{equation}\label{ineq:GF}
\norm{u}_{H^1(\Omega)^3} \leq C \|u\|_{\xn^\varepsilon (\Omega)} \quad \forall u \in \xn^\varepsilon (\Omega).
\end{equation}
As a consequence, the space $\xn^\varepsilon(\Omega)$ is continuously embedded into $H^1(\Omega)^3$.
We refer to Prokhorov and Filonov \cite[Thm. 1.1]{PrFi15} for a proof of \eqref{ineq:GF}.  Their proof includes 
the case of convex domains or in general Lipschitz domains satisfying the exterior ball condition, and they require less regular permittivities than the ones in $\mathcal{E}$.
 Another  proof can be found in Alberti and Capdeboscq \cite{AlCa14}. Classical references for the Gaffney inequality are   Saranen \cite{Sa83} and Mitrea \cite{Mit01}. More recently, Creo and Lancia \cite{CrLa20} proved a weaker form of the Gaffney inequality in more irregular domains in dimension $2$ and $3$.

We now formulate  the assumptions on the set $\Omega$ that  we consider  in some of the main results of the paper. Given $\alpha,\beta,\gamma>0$  with $0 < \alpha < \beta$, then    
\begin{eqnarray}\label{Omega_def}
& 
\text{$\Omega$ is a bounded Lipschitz domain of $\mathbb{R}^3$ such that  the Gaffney inequality \eqref{ineq:GF}}\nonumber \\ 
& \text{ holds  with a   constant $C>0$ independent of $\varepsilon\in \mathcal{A}_{\alpha,\beta,\gamma}$.  } 
\end{eqnarray}   
\begin{rem}
If the bounded Lipschitz domain $\Omega$ is of class $C^{1,1}$ or convex,  the requirement in \eqref{Omega_def} is automatically satisfied. Indeed by looking at the proof of the Gaffney inequality (see e.g. \cite{PrFi15}), one can realize that that under either one of these assumptions on $\Omega$, if $\varepsilon \in   \mathcal{A}_{\alpha,\beta,\gamma}$ then the constant $C$ in \eqref{ineq:GF} depends only on $\alpha, \beta $ and on  the $L^\infty$-norm of the derivatives of $\varepsilon$.
\end{rem}

We now turn our attention to the eigenvalue problem \eqref{prob:eig}.
The weak formulation of problem \eqref{prob:eig}, by standard integration by parts (see e.g. Kirsch and Hettllich \cite[Thm. A.13]{kihe}, is the following:
\begin{equation}\label{prob:eigen1weak}
\int_{\Omega} \mathrm{curl}\,u \cdot \mathrm{curl}\, v \,dx  =\lambda\int_{\Omega}\varepsilon u\cdot v \,dx
  \qquad \forall v \in \xn^\varepsilon(\mathrm{div}\,\varepsilon \, 0,\Omega),
\end{equation}
in the unknowns $\lambda \in \mathbb{R}$ (the eigenvalues) and $u \in \xn^\varepsilon(\mathrm{div}\,\varepsilon \, 0,\Omega)$ (the eigenvectors).
Observe that the eigenvalues $\lambda$ of problem \eqref{prob:eigen1weak} are non-negative, as one can easily see by setting $v=u$ in  \eqref{prob:eigen1weak}.  

For our purposes it will be convenient  to work in the space $\xn^\varepsilon(\Omega)$ rather than $\xn^\varepsilon(\mathrm{div}\,\varepsilon\, 0,\Omega)$. Hence, following Costabel \cite{Co91} and Costabel and Dauge \cite{CoDa99}, we consider the following eigenvalue problem which presents  an additional penalty term
\begin{equation}\label{prob:eigen2weak}
\int_{\Omega} \mathrm{curl}\,u \cdot \mathrm{curl}\,v \,dx + \tau \int_{\Omega} \operatorname{div}(\varepsilon u )\, \operatorname{div}(\varepsilon v) \,dx = \sigma \int_{\Omega}\varepsilon u\cdot v \,dx \quad \forall\, v \in \xn^\varepsilon(\Omega),
\end{equation}
in the unknowns $u \in  \xn^\varepsilon(\Omega)$ and $\sigma \in \mathbb{R}$.
Here $\tau>0$ is any fixed positive real number.
Solutions of problem \eqref{prob:eigen1weak} will then corresponds to solutions $u$ of \eqref{prob:eigen2weak} with $\operatorname{div}(\varepsilon u)=0$ in $\Omega$ (see Theorem \ref{thm:codau} below). 
We define the following quadratic form on $\xn^\varepsilon(\Omega)$, equivalent to its inner product:
\begin{equation}\label{def:Te}
T_\varepsilon [u,v] :=   \int_\Omega \varepsilon \, u \cdot v \, dx +  \int_\Omega \operatorname{curl}u \cdot  \operatorname{curl}v \,dx + \tau \int_\Omega \operatorname{div}(\varepsilon u) \operatorname{div}(\varepsilon v)\, dx \qquad 
\forall\, u,v \in \xn^\varepsilon(\Omega).
\end{equation}
The eigenvalue problem \eqref{prob:eigen2weak} is then equivalent to
\begin{equation}\label{prob:eigen3weak}
 T_\varepsilon [u,v] = (\sigma+1)\langle u,v \rangle_\varepsilon \qquad \forall \,v  \in \xn^\varepsilon(\Omega).
 \end{equation}

If $\Omega$ is a bounded Lipschitz domain and $\varepsilon \in \mathcal{E}$, the space $\xn^\varepsilon(\Omega)$ is compactly embedded into $L^2(\Omega)^3$ (see also Weber \cite{We80}). Hence we deduce that the spectrum of problem \eqref{prob:eigen2weak} is composed by non-negative eigenvalues of finite multiplicity (depending on $\varepsilon$) which can be arranged in a non-decreasing divergent sequence
\[
0\leq \sigma_1[\varepsilon] \leq \sigma_2[\varepsilon] \leq \cdots \leq \sigma_n[\varepsilon] \leq \cdots \nearrow +\infty.
\]
Here each eigenvalue is repeated in accordance with its multiplicity. It is important to observe that the zero eigenspace has dimension equal to the number of connected components of $\partial  \Omega$ minus one (see e.g. Assous, Ciarlet and Labrounie \cite[Prop. 6.1.1]{AsCiLa18}). In particular, the multiplicity of the zero eigenvalue depends only on the topology of $\Omega$.  
By the spectral theorem, we also have a standard min-max variational characterization of each eigenvalue:
\begin{equation} \label{minmax:formula}
\sigma_j[\varepsilon] = \min_{\substack{V_j \subset \xn^\varepsilon(\Omega),\\ \operatorname{dim}V_j = j}} \max_{\substack{u \in V_j,\\u \neq 0}} \frac{\int_{\Omega} \abs{\operatorname{curl}u}^2 dx + \tau \int_\Omega \abs{\operatorname{div}(\varepsilon u)}^2 dx}{\int_{\Omega} \varepsilon u \cdot u \, dx}.
\end{equation}
In view of the Gaffney inequality, formula \eqref{minmax:formula} can be rewritten as follows:
\begin{equation} \label{minmax:formula:H1}
\sigma_j[\varepsilon] = \min_{\substack{V_j \subset H^1_{\rm \scriptscriptstyle N}(\Omega),\\ \operatorname{dim}V_j = j}} \max_{\substack{u \in V_j,\\u \neq 0}} \frac{\int_{\Omega} \abs{\operatorname{curl}u}^2 dx +\tau \int_\Omega \abs{\operatorname{div}(\varepsilon u)}^2 dx}{\int_{\Omega} \varepsilon u \cdot u \, dx},
\end{equation}
where
\begin{equation*}
H^1_{\rm \scriptscriptstyle N}(\Omega) : = \set{u \in H^1(\Omega)^3 :  \nu \times u\rvert_{\partial \Omega}=0}.
\end{equation*}
Additionally, it is possibile to characterize the eigenvalues of problem \eqref{prob:eigen2weak} by means of  two families of eigenvalues: one corresponding to the original problem 
\eqref{prob:eigen1weak} (or \eqref{prob:eig}), and one associated with the operator $-\operatorname{div}(\varepsilon \nabla \cdot)$ with Dirichlet boundary conditions. More precisely, we recall the following result in the same spirit of  Costabel and Dauge \cite[Thm 1.1]{CoDa99}
(see also \cite[Thm. 2.2]{LuZa22} for a proof). 
\begin{theorem} \label{thm:codau}   
Let $\Omega$ be a bounded Lipschitz domain and  $\varepsilon \in \mathcal{E}$ be such that \eqref{ineq:GF} holds.
Then the eigenpairs $(\sigma, u) \in \mathbb{R} \times \xn^\varepsilon(\Omega)$ of problem 
\eqref{prob:eigen2weak} are given by the following two disjoint families:
\begin{enumerate}[label=\normalfont\roman*)]
\item the pairs $( \lambda,  u)   \in \mathbb{R} \times \xn^\varepsilon(\mathrm{div}\,\varepsilon \, 0,\Omega)$ solutions of  problem   \eqref{prob:eigen1weak};
\item the pairs $(\tau \rho,\nabla f)$ where $(\rho,f) \in \mathbb{R} \times  H^1_0(\Omega)$ is an eigenpair of the problem
\begin{equation} \label{dirichlet:problem:simillap}
\begin{cases}
-\operatorname{div}(\varepsilon \nabla f)=\rho f & \text{in }\Omega,\\
f=0 & \text{on }\partial \Omega.
\end{cases}
\end{equation}
\end{enumerate}
In particular, the set of eigenvalues of problem  \eqref{prob:eigen2weak} are given by the union of the set of eigenvalues of problem \eqref{prob:eigen1weak} and the set of eigenvalues of the operator 
$-\operatorname{div}(\varepsilon\nabla\cdot)$ with Dirichlet boundary conditions  in $\Omega$ multiplied by  $\tau$.
\end{theorem}
In view of Theorem \ref{thm:codau}, we call {\it Maxwell eigenvalues} and {\it Maxwell eigenvectors} the eigenpairs corresponding to the first family in i). More precisely we give the following definition.
\begin{defn}
Let $\Omega$ be a bounded Lipschitz domain and  $\varepsilon \in \mathcal{E}$ be such that \eqref{ineq:GF} holds.
An eigenvalue $\sigma$ of problem \eqref{prob:eigen2weak} is said to be 
  a \emph{Maxwell eigenvalue}  with permittivity $\varepsilon$ if  there exists $u \in \xn^\varepsilon(\mathrm{div}\,\varepsilon \, 0,\Omega)$, $u \neq 0$, such that $(\sigma, u)  $ is an eigenpair of problem \eqref{prob:eigen1weak}.
  In this case, we say that $u$ is a \emph{Maxwell eigenvector}. We denote the set of Maxwell eigenvalues by:
 \[
0\leq \lambda_1[\varepsilon] \leq \lambda_2[\varepsilon] \leq \cdots \leq \lambda_n[\varepsilon] \leq \cdots \nearrow +\infty,
\]
where we repeat the eigenvalues in accordance with their (Maxwell) multiplicity,  i.e. the dimension of the space generated by the corresponding Maxwell eigenvectors. 
\end{defn}

 \section{Weak* continuity of  Maxwell eigenvalues}\label{sec:weak}

In \cite[Thm. 3.2]{LuZa22} it is  proved that the Maxwell eigenvalues are locally Lipschitz continuous with respect to 
$\varepsilon \in \mathcal{E}$, and thus in particular they are continuous. 
Here we show that the Maxwell eigenvalues $\lambda_j[\varepsilon]$ depend with continuity on $\varepsilon$ not only with respect to the strong topology of $W^{1,\infty}(\Omega)^{3 \times 3}$, but also with respect to its weak* topology. Note that the continuity in the weak* topology is more relevant in view of applications to optimization problems.

We now recall what we mean by the  weak* topology of $W^{1,\infty}(\Omega)$.
We consider $W^{1,\infty}(\Omega)$ as a subspace of $L^\infty(\Omega)^4$.
The inclusion is interpreted by identifying a function $f \in W^{1,\infty}(\Omega)$ as the quadruple 
\[
(f, \partial_{x_1} f, \partial_{x_2} f, \partial_{x_3} f) = (f, \nabla f) \in L^\infty(\Omega) \times L^\infty(\Omega)^3= L^\infty(\Omega)^4.
\] 
Then we endow $W^{1,\infty}(\Omega)$ with the topology induced by  the weak* topology  of $L^\infty(\Omega)^4$.
We call this topology the  weak*  topology of $W^{1,\infty}(\Omega)$.

Note that the subspace $W^{1,\infty}(\Omega)$ is sequentially weakly* closed in $L^\infty(\Omega)^4$.  Indeed it is not difficult to see that if $\{f_k\}_{k \in \mathbb{N}}
\subseteq W^{1,\infty}(\Omega)$ converges weakly* in $L^\infty(\Omega)^4$ to some quadruple $(f,F)\in L^\infty(\Omega)^4$, that is 
\begin{equation*} 
\int_\Omega f_k g \, dx + \int_\Omega \nabla f_k \cdot  G \, dx \to \int_\Omega f g \, dx + \int_\Omega F \cdot G \, dx \quad \mbox{ as } k \to +\infty
\end{equation*}
for all $(g, G) \in  L^1(\Omega)^4$, then $f \in W^{1,\infty}(\Omega)$ and $\nabla f = F$.
Sequential weak* closedness is sufficient for our purposes, since we will work with bounded sets in $W^{1,\infty}(\Omega)$ (see the definition of $\mathcal{A}_{\alpha,\beta,\gamma}$ in \eqref{def:Aabc}),  which are metrizable with respect to the weak* topology due to the fact that $L^1(\Omega)^4$ is separable (see e.g. Megginson \cite[Thm. 2.6.23]{meg}).
Hence in the sequel the notions of sequentially weakly* closed and weakly* closed coincide.
We summarize the above discussion in the following definition.

\begin{defn} \label{defn:weakstar:conv}
We say that a sequence of  functions $\{f_k\}_{k \in \mathbb{N}} \subseteq W^{1,\infty}(\Omega)$ weakly* converges to a function $f \in W^{1,\infty}(\Omega)$, if for all $(g, G) \in L^1(\Omega)^4$ one has
\begin{equation} \label{type:weakstar:conv}
\int_\Omega f_k g \, dx + \int_\Omega \nabla f_k \cdot G \, dx \to \int_\Omega f g \, dx + \int_\Omega \nabla f \cdot G \, dx \quad \mbox{ as } k \to +\infty. 
\end{equation}
 In this case we write  $f_k \rightharpoonup^* f$ in $W^{1,\infty}(\Omega)$.
 We say that a sequence of $(3 \times 3)$-matrix valued maps $\{M_k\}_{k \in \mathbb{N}} \subseteq W^{1,\infty}(\Omega)^{3 \times 3}$ weakly* converges to a map $M \in W^{1,\infty}(\Omega)^{3 \times 3}$ if \eqref{type:weakstar:conv} is satisfied for every $(i,j)$-component  of the matrix for all $i,j=1,2,3$. In this case we  write  $M_k \rightharpoonup^* M$ in $W^{1,\infty}(\Omega)^{3 \times 3}$.
 
\end{defn}

\begin{rem} \label{strLinf}
Observe that, by the uniform boundedness principle, if $\{f_k\}_{k \in \mathbb{N}} \subseteq W^{1,\infty}(\Omega)$ is such that $f_k \rightharpoonup^* f$ in $W^{1,\infty}(\Omega)$, then the sequence is bounded in $W^{1,\infty}(\Omega)$.
Since $\Omega \subset \mathbb{R}^3$ is bounded, then $W^{1,\infty}(\Omega)$ is contained in $W^{1,p}(\Omega)$ for any $p \geq 1$.  Under the assumption that $\Omega$ is a bounded Lipschitz domain, taking $p>3$ and using  the Rellich-Kondrachov embedding theorem (see, e.g.,  Adams and Fournier \cite[Thm. 6.3]{ad03}), we have that
we can also assume that, up to passing to a subsequence,  $f_k \to f$ strongly  in $L^\infty(\Omega)$.
\end{rem}

Let $\alpha,\beta,\gamma>0$   such that $0 < \alpha < \beta$. For the sake of the reader, we recall the definition of the restricted class of permittivities $\mathcal{A}_{\alpha,\beta,\gamma}$ (cf. \eqref{def:Aabc}):
\begin{equation*}
\begin{split}
\mathcal{A}_{\alpha,\beta,\gamma} := \bigg\{\varepsilon & \in W^{1,\infty}(\Omega)^{3\times 3} \cap \mathrm{Sym}_3(\mathbb{R}) : \alpha |\xi|^2  \leq \varepsilon (x) \xi \cdot \xi \leq \beta |\xi|^2 \text{ for a.e. }x \in \Omega, \\
& \text{ for all } \xi \in \mathbb{R}^3, \ \norm{\nabla \varepsilon_{ij}}_{L^\infty(\Omega)^3} \leq \gamma \text{ for all } 1\leq i,j \leq 3\bigg\}.
\end{split}
\end{equation*}
Before proceeding any further, we first recall that for a fixed permittivity $\varepsilon$
 all the eigenvalues $\{\sigma_j[\varepsilon]\}_{j \in \mathbb{N}}$ can be  uniformly bounded  from above by means of the eigenvalues
  $\{\sigma_j[\mathbb{I}_3]\}_{j \in \mathbb{N}}$, where $\sigma_j[\mathbb{I}_3]$ is the $j$-th eigenvalue of problem \eqref{prob:eigen2weak} with unitary permittivity. The proof exploits formula  \eqref{div:matrixxvector} together with the Gaffney inequality and can be found in last part of the proof of 
 \cite[Thm. 32]{LuZa22}.
\begin{lem} \label{bounds:eigenvalues}
 Let $\alpha,\beta,\gamma>0$  with $0 < \alpha < \beta$. Let $\Omega$ be as in \eqref{Omega_def}.    Then  there exists a constant $C>0$  such that
\[
\sigma_j[\varepsilon] \leq C (\sigma_j[\mathbb{I}_3]+1) \qquad \forall j \in \mathbb{N}, \,\forall \varepsilon \in \mathcal{A}_{\alpha,\beta,\gamma}.
\]
\end{lem}

We can now state the main result of this section. The following Theorem \ref{weak*:cont:thm} is in the flavour of  Cox and McLaughlin \cite[Proposition~4.3]{CoMc90A} 
which deals with  Dirichlet Laplacian eigenvalues   in the presence of a mass density  parameter modeling non-homogeneous membranes (cf. problem \eqref{massden}). See also \cite[Thm. 3.1]{LaPr13} where the authors extend \cite[Thm. 9.1.3]{He06}  to general elliptic operators of arbitrary order subject to homogeneous boundary conditions. We also refer to  \cite[Thm. 9.1.3]{He06}.  

\begin{theorem} \label{weak*:cont:thm}
 Let $\alpha,\beta,\gamma>0$  with $0 < \alpha < \beta$. Let $\Omega$ be as in \eqref{Omega_def}.
The map $\varepsilon \mapsto \lambda_j[\varepsilon]$ from $\mathcal{A}_{\alpha,\beta,\gamma}$ to $\mathbb{R}$ is weakly* continuous for all $j \in \mathbb{N}$.
\end{theorem}
\begin{proof}
Let $\varepsilon \in  \mathcal{A}_{\alpha,\beta,\gamma}$ and $\{\varepsilon_k\}_{k \in \mathbb{N}} \subseteq \mathcal{A}_{\alpha,\beta,\gamma}$. We need to show that if $\varepsilon_k \rightharpoonup^* \varepsilon$ in 
$W^{1,\infty}(\Omega)^{3 \times 3}$ as $k \to +\infty$, then $\lambda_j[\varepsilon_k] \to \lambda_j[\varepsilon]$ as $k \to +\infty$ for all $j \in \mathbb{N}$. 
Observe that for any $j \in \mathbb{N}$, thanks to Lemma \ref{bounds:eigenvalues},  the sequence $\{\lambda_j[\varepsilon_k]\}_{k \in \mathbb{N}}$ is bounded. That is, for all $j \in \mathbb{N}$ there exists  $L_j > 0$ such that 
\begin{equation} \label{boundedness:sigmajk}
\lambda_j[\varepsilon_k] \leq L_j \quad \forall \,k \in \mathbb{N}.
\end{equation} 
Let $\{u_j[\varepsilon_k]\}_{j \in \mathbb{N}}$ be an $L^2_{\varepsilon_k}(\Omega)$-orthonormal sequence of Maxwell eigenfunctions, where each $u_j[\varepsilon_k]$ 
corresponds to the eigenvalue $\lambda_j[\varepsilon_k]$. In particular,  $\operatorname{div}(\varepsilon_k u_j[\varepsilon_k])=0$  and thus
the weak formulation \eqref{prob:eigen2weak} reads as follows:
\begin{equation} \label{weak:variationalform:k}
\int_\Omega \operatorname{curl}u_j[\varepsilon_k]\cdot \operatorname{curl} v \, dx = \lambda_j[\varepsilon_k] \int_\Omega \varepsilon_k \, u_j[\varepsilon_k] \cdot v \, dx \quad \forall\, v \in \xn^{\varepsilon_k}(\Omega).
\end{equation}
Moreover, since    $\int_\Omega \varepsilon_k \, u_m[\varepsilon_k] \cdot u_n[\varepsilon_k] \, dx = \delta_{mn}$ for all $m,n\in \mathbb{N}$, we have
\begin{equation*}
\int_\Omega \abs{u_j[\varepsilon_k]}^2 \, dx  \leq \frac{1}{\alpha} \int_\Omega \varepsilon_k u_j[\varepsilon_k] \cdot u_j[\varepsilon_k] \, dx = \frac{1}{\alpha}
\end{equation*}
and
\begin{equation*}
\int_\Omega \abs{\operatorname{curl}u_j[\varepsilon_k]}^2 dx = \lambda_j[\varepsilon_k] \int_\Omega \varepsilon_k \, u_j[\varepsilon_k] \cdot u_j[\varepsilon_k] \,  dx \leq L_j
\end{equation*}
for all $k \in \mathbb{N}$.

By our assumptions  on $\Omega$, there exists a constant $C>0$ such that   
\begin{equation*} \|u\|_{H^1(\Omega)^3} \leq C \|u\|_{\xn^{\varepsilon_k}(\Omega)}
\end{equation*}
for all  $k \in \mathbb{N}$,  $u \in \xn^{\varepsilon_k}(\Omega)$. Note that $C$ is independent of $k$.

In particular, the sequence $\{u_j[\varepsilon_k]\}_{k \in \mathbb{N}}$ is bounded in $H^1_{\rm \scriptscriptstyle N}(\Omega)$. 
The Banach-Alaoglu theorem implies that, possibly passing to a subsequence, there exists $\bar{u}_j \in H^1_{\rm \scriptscriptstyle N}(\Omega)$ such that  $u_j[\varepsilon_k]$ weakly converges to $\bar{u}_j$ as $k \to+ \infty$ in $H^1_{\rm \scriptscriptstyle N}(\Omega)$. 
Also, again up to considering subsequences, by \eqref{boundedness:sigmajk}  there exists $\bar{\lambda}_j \in \mathbb{R}$ such that $\lambda_j[\varepsilon_k] \to \bar{\lambda}_j$ as $k \to+ \infty$. 
Moreover, since the embedding of $H^1_{\rm \scriptscriptstyle N}(\Omega)$ into $L^2(\Omega)^3$ is compact (see e.g. Weber \cite{We80}) we can also assume that $u_j[\varepsilon_k] \to \bar{u}_j$ strongly in $L^2(\Omega)^3$ as $k \to+ \infty$.

Since the spaces $H^1_{\rm \scriptscriptstyle N}(\Omega)$ and $\xn^{\varepsilon_k}(\Omega)$ coincides algebraically and topologically, passing to the limit as $k \to +\infty$ in \eqref{weak:variationalform:k} we get that
\begin{equation} \label{uj:limit:equation}
\int_\Omega \operatorname{curl}\bar{u}_j \cdot \operatorname{curl}v \, dx = \bar{\lambda}_j \int_\Omega \varepsilon \, \bar{u}_j \cdot v \, dx \quad \forall\, v \in  H^1_{\rm \scriptscriptstyle N}(\Omega).
\end{equation}
The limit in the left-hand side is due to the fact that  $u_j[\varepsilon_k] \rightharpoonup \bar{u}_j$ weakly in $H^1_{\rm \scriptscriptstyle N}(\Omega)$ as $k \to +\infty$. 
In order to prove the limit on the right-hand side, we note that
\begin{equation*}
\begin{split}
&\abs{\int_\Omega (\varepsilon_k u_j[\varepsilon_k] - \varepsilon \bar{u}_j) \cdot v dx} \leq \abs{\int_\Omega \varepsilon_k (u_j[\varepsilon_k] - \bar{u}_j) \cdot v \, dx} + \abs{\int_\Omega (\varepsilon_k - \varepsilon) \, \bar{u}_j \cdot v \, dx}\\
& \leq 3 \norm{\varepsilon_k}_{L^\infty(\Omega)^{3 \times 3}} \norm{u_j[\varepsilon_k] - \bar{u}_j}_{L^2(\Omega)^3} \norm{v }_{L^2(\Omega)^3} + \abs{\int_\Omega (\varepsilon_k - \varepsilon) \, \bar{u}_j \cdot v \, dx}.
\end{split}
\end{equation*}
The first term of the sum goes to 0 as $k \to +\infty$ since $u_j[\varepsilon_k] $ strongly converges to $\bar{u}_j$  in $L^2(\Omega)^3$, while the second term goes to 0 since $\varepsilon_k \rightharpoonup^* \varepsilon$ in $W^{1,\infty}(\Omega)^{3 \times 3}$.

Furthermore, we have that $\operatorname{div}(\varepsilon \bar{u}_j)=0$. Indeed fix any $\eta \in C^\infty_c(\Omega)$. 
Then 
\begin{equation*}
\int_\Omega \varepsilon_k \, u_j[\varepsilon_k] \cdot \nabla \eta \, dx =0 \qquad \forall\,k \in \mathbb{N}
\end{equation*}
since $\operatorname{div}(\varepsilon_k u_j[\varepsilon_k])=0$.
Taking the limit as $k \to +\infty$ and reasoning as above, we get that
\begin{equation*}
\int_\Omega \varepsilon \, \bar{u}_j \cdot \nabla \eta=0,
\end{equation*}
that is $\operatorname{div}(\varepsilon \bar{u}_j)=0$ in $\Omega$.
Therefore from \eqref{uj:limit:equation} it follows that $\bar{u}_j \in \xn^\varepsilon(\operatorname{div}\varepsilon \, 0, \Omega)$ is a Maxwell eigenfunction of problem \eqref{prob:eigen2weak} corresponding to the Maxwell eigenvalue $\bar{\lambda}_j$.

Observe that taking the limit in the orthonormality condition of the eigenfunction we get $\langle\bar{u}_j, \bar{u}_l\rangle_\varepsilon = \delta_{jl}$ for all $j,l \in \mathbb{N}$, and then $\{\bar{\lambda}_j\}_{j \in \mathbb{N}}$ is a divergent sequence.

In order to prove that $\bar{\lambda}_j = \lambda_j[\varepsilon]$ for all $j \in \mathbb{N}$, we assume by contradiction that  there exists a Maxwell eigenfunction $\bar{u} \in \xn^\varepsilon(\Omega)$ of problem \eqref{prob:eigen2weak}, with $\operatorname{div}(\varepsilon \bar{u})=0$, corresponding to a Maxwell eigenvalue $\bar{\lambda}$, such that it is orthogonal in $L^2_\varepsilon(\Omega)$ to all the $\bar{u}_j$, $j \in \mathbb{N}$, namely that $\langle\bar{u},\bar{u}_j\rangle_\varepsilon  =0$ for all $j \in \mathbb{N}$.

Before proceeding any further, we note that the Maxwell's eigenvalues do not depend on the choice of the penalty term $\tau$. Then, for any $j,k \in \mathbb{N}$ fixed, we
can chose $\tau$ big enough so that  the first $j$ eigenvalues of problem \eqref{prob:eigen2weak} are all of  Maxwell type, i.e. without resonances with eigenfunctions deriving from problem \eqref{dirichlet:problem:simillap} (cf. Theorem \ref{thm:codau}). Doing this permits us to apply Theorem \ref{thm:min:f:M+1} to Maxwell eigenvalues, even if it is stated for the eigenvalues of 
 problem \eqref{prob:eigen2weak}.

Assume, without losing generality, that $\bar{u}$ is such that 
\[
\norm{\bar{u}}_{L^2_\varepsilon(\Omega)}=1/(\bar{\lambda} +1).
\]
 At this point one  wants to apply the Auchmuty principle of Theorem \ref{thm:min:f:M+1} with
$\bar{u}$ and then pass to the limit as $k \to +\infty$ to get the contradiction. However one faces the problem that the weak* convergence of $\varepsilon_k$ does not suffice to obtain the convergence of the penalty term $\int_\Omega (\operatorname{div}(\varepsilon_k \bar{u}))^2\,dx$. To solve this difficulty we approximate $\bar{u}$ by a sequence of functions   $\{\bar{u}_k\}_{k \in \mathbb{N}}$ with $\bar{u}_k \in  \xn^\varepsilon(\operatorname{div}\varepsilon_k \, 0, \Omega)$ to get rid of the penalty term.

To this aim we consider the problem
\[
\begin{cases}
\operatorname{div}\left(\varepsilon_k \nabla\varphi\right) = \operatorname{div}(\varepsilon_k\bar{u}) \quad &\mbox{ in } \Omega,\\
{\varphi}  = 0 \quad &\mbox{ on } \partial \Omega,
\end{cases}
\]
which is understood in the weak sense:
\begin{equation}\label{eq:weakphi}
\int_{\Omega}\varepsilon_k \nabla\varphi \cdot \nabla\psi \,dx = \int_{\Omega}\varepsilon_k \bar{u} \cdot \nabla\psi \, dx \qquad \forall \psi \in 
H^1_0(\Omega).
\end{equation}
Let $\varphi_k \in H^1_0(\Omega)$ be the unique solution of the above problem \eqref{eq:weakphi}. Note that the sequence 
$\{\varphi_k \}_{k \in \mathbb{N}}$ is bounded in $H^1_0(\Omega)$, indeed:
\begin{align*}
\alpha \|\varphi_k\|^2_{H^1_0(\Omega)} \leq \left| \int_{\Omega}\varepsilon_k \nabla\varphi_k \cdot \nabla\varphi_k \,dx \right|
&= \left|\int_{\Omega}\varepsilon_k \bar{u} \cdot \nabla \varphi_k \, dx \right| \\
& \leq 3 \|\varepsilon_k\|_{L^\infty(\Omega)^{3 \times 3}}
 \|\varphi_k\|_{H^1_0(\Omega)}  \|\bar{u}\|_{L^2(\Omega)}, 
\end{align*}
so that 
\[
 \|\varphi_k\|_{H^1_0(\Omega)} \leq \frac{3}{\alpha}\|\varepsilon_k\|_{L^\infty(\Omega)^{3 \times 3}} \|\bar{u}\|_{L^2(\Omega)}.
\]
Then, up to passing to a subsequence, $\varphi_k$ weakly converges to some $\hat \varphi$ in $H^1_0(\Omega)$.

We  prove that $\hat\varphi = 0$.  We pass to the limit in the equation \eqref{eq:weakphi} solved by $\varphi_k$. 
Since $\operatorname{div}(\varepsilon\bar{u})=0$, we observe that 
for all $\psi \in H^1_0(\Omega)$:
\[
 \int_{\Omega}\varepsilon_k \bar{u} \cdot \nabla\psi \, dx  \to \int_{\Omega}\varepsilon \bar{u} \cdot \nabla\psi \, dx  = 0 \quad \mbox{ as } k\to +\infty.
\]
Also, 
\[
 \int_{\Omega}\varepsilon_k  \nabla \varphi_k \cdot \nabla\psi \, dx  \to \int_{\Omega}\varepsilon  \nabla \hat\varphi \cdot \nabla\psi \, dx  \quad\mbox{ as } k\to +\infty,
\]
indeed 
\begin{align}\label{ineq1}
 \Bigg|\int_{\Omega}&\varepsilon_k  \nabla \varphi_k \cdot \nabla\psi \, dx -  \int_{\Omega}\varepsilon  \nabla\hat \varphi \cdot \nabla\psi \, dx \Bigg| \\ \nonumber
 &\leq  \left|\int_{\Omega}(\varepsilon_k-\varepsilon)  \nabla \varphi_k \cdot \nabla\psi \, dx \right| + \left|\int_{\Omega} \varepsilon\left(  \nabla \varphi_k-\nabla \hat\varphi\right) \cdot \nabla\psi \, dx \right| 
\end{align}
As $k \to +\infty$, the first term in the right-hand side  converges to zero by the strong convergence of $\varepsilon_k$ in $L^\infty(\Omega)^{3 \times 3}$ and by the boundedness of $\varphi_k$ in $H_0^1(\Omega)$, while the second term converges to zero by the weak convergence of $\varphi_k$ in  $H_0^1(\Omega)$. This shows that $\hat \varphi$ solves in the weak sense the problem 
\[
\begin{cases}
\operatorname{div}\left(\varepsilon \nabla\hat \varphi\right) = 0 \quad &\mbox{ in } \Omega,\\
{\hat \varphi}  = 0 \quad &\mbox{ on } \partial \Omega,
\end{cases}
\]
and accordingly $\hat \varphi = 0$.

Actually we can prove that $\varphi_k$ converges in norm to zero using again the equation:
 \begin{align}\label{ineq2}
 \alpha \|\varphi_k\|^2_{H^1_0(\Omega)} \leq \|\nabla \varphi_k\|_{L^2_{\varepsilon_k}(\Omega)} = \int_{\Omega}\varepsilon_k  \nabla \varphi_k \cdot \nabla \varphi_k \, dx  = \int_{\Omega}\varepsilon_k \bar{u} \cdot \nabla \varphi_k \, dx.
 \end{align}
 The right-hand side   converges to zero by reasoning as in \eqref{ineq1}, and then $\|\varphi_k\|_{H^1_0(\Omega)} \to 0$ as $k$ goes to $+\infty$. 
 
 Now we set
 \[
 \bar{u}_k := \bar{u} - \nabla \varphi_k \qquad \forall k \in  \mathbb{N}.
 \]
 Note that, by construction, $\operatorname{div}(\varepsilon_k\bar{u}_k)=0$ and $\operatorname{curl}(\bar{u}_k) = \operatorname{curl}(\bar{u})$ for all $k \in \mathbb{N}$.
By  Theorem \ref{thm:min:f:M+1}  we get
 
 \begin{equation} \label{auchmuty:inequality}
-\frac{1}{2(\lambda_j[\varepsilon_k] +1)} \leq \frac{1}{2}T_{\varepsilon_k}[ \bar{u}_k, \bar{u}_k] - \norm{ \bar{u}_k- P_{j-1,\varepsilon_k}  \bar{u}_k}_{L^2_{\varepsilon_k}(\Omega)},
\end{equation}
 where  $T_\varepsilon$ is the operator defined in \eqref{def:Te} and $P_{j-1,\varepsilon_k}$ denotes the orthogonal projection in $L^2_{\varepsilon_k}(\Omega)$   onto the linear span of $u_1[\varepsilon_k], \dots, u_{j-1}[\varepsilon_k]$.

 We first claim that 
 \begin{equation}\label{limit1}
 T_{\varepsilon_k}[\bar{u}_k,\bar{u}_k] \to T_\varepsilon [\bar{u},\bar{u}] \quad \mbox{ as } k \to +\infty.
 \end{equation}
Indeed
 \[
  T_{\varepsilon_k}[\bar{u}_k,\bar{u}_k]  =  \int_\Omega \varepsilon_k \, \bar{u}_k \cdot \bar{u}_k \, dx +  \int_\Omega |\operatorname{curl}\bar{u}|^2   \,dx,
 \]
 and 
 \begin{align*}
  \int_\Omega \varepsilon_k \, \bar{u}_k \cdot \bar{u}_k \, dx &= \int_\Omega \varepsilon_k \, (\bar{u} - \nabla \varphi_k )  \cdot (\bar{u} - \nabla \varphi_k ) \, dx\\
  &= \int_\Omega \varepsilon_k \, \bar{u} \cdot \bar{u} \, dx -2  \int_\Omega \varepsilon_k \, \bar{u}    \cdot  \nabla \varphi_k  \, dx
  + \int_{\Omega}\varepsilon_k  \nabla \varphi_k \cdot \nabla \varphi_k \, dx.
 \end{align*}
 The last two terms in the right hand side   converge to zero by  previous estimates (cf. \eqref{ineq2}), so that 
 \[
  \int_\Omega \varepsilon_k \, \bar{u}_k \cdot \bar{u}_k \, dx \to  \int_\Omega \varepsilon \, \bar{u} \cdot \bar{u} \, dx \quad \mbox{ as } k \to +\infty
 \]
 and thus \eqref{limit1} holds.
 
 Finally we need to show that 
 \begin{equation}\label{ineq3}
 \norm{ \bar{u}_k- P_{j-1,\varepsilon_k}  \bar{u}_k}_{L^2_{\varepsilon_k}(\Omega)} \to \|\bar u\|_{L^2_{\varepsilon}(\Omega)} \quad \mbox{ as } k \to +\infty.
 \end{equation}
 To this aim we note that  
 \[
  \norm{ \bar{u}_k-\bar u  }_{L^2_{\varepsilon_k}(\Omega)} = \|\nabla \varphi_k\|_{L^2_{\varepsilon_k}(\Omega)} \to 0 \quad \mbox{ as } k \to +\infty.
 \]
 Also,
 \[
 \left\|P_{j-1,\varepsilon_k}  \bar{u}_k\right\|_{L^2_{\varepsilon_k}(\Omega)} 
 = \left\| \sum_{i=1}^{j-1}\langle \bar{u}_k,u_i[\varepsilon_k] \rangle_{\varepsilon_k }u_i[\varepsilon_k]\right\|_{L^2_{\varepsilon_k}(\Omega)}  \to 0 \quad \mbox { as } k \to +\infty.
 \]
Indeed $u_i[\varepsilon_k]$ converges strongly to  $\bar{u}_i$  in $L^2(\Omega)^3$ and 
 \[
\langle \bar{u}_k,u_i[\varepsilon_k] \rangle_{\varepsilon_k } 
= \langle \bar{u} - \nabla \varphi_k,u_i[\varepsilon_k] \rangle_{\varepsilon_k }
=\langle \bar{u} ,u_i[\varepsilon_k] \rangle_{\varepsilon_k } - \langle \nabla \varphi_k,u_i[\varepsilon_k] \rangle_{\varepsilon_k }.
  \]
 As $k \to +\infty$,  the first term in the right-hand side converges to zero since $\langle \bar{u}, \bar{u}_i \rangle_\varepsilon =0$ for all $i \in \mathbb{N}$, while the second converges to zero by the strong convergence of $\nabla \varphi_k$ to zero in $L^2(\Omega)^3$.
 Thus \eqref{ineq3} is proved.
  
  We are now ready to pass to the limit as $k \to +\infty$ in  inequality \eqref{auchmuty:inequality} and obtain that for all $j \in \mathbb{N}$:
  \begin{align} \label{contradiction:auchmuty}
-\frac{1}{2(\bar{\lambda}_j +1)}& \leq \frac{1}{2}T_{\varepsilon}[\bar{u}][\bar{u}] - \norm{\bar{u}}_{L^2_{\varepsilon}(\Omega)}\\ \nonumber
& =\frac{1}{2(\bar \lambda +1)}  -\frac{1}{(\bar{\lambda}+1)}= -\frac{1}{2(\bar{\lambda}+1)} <0.
\end{align}
Inequality \eqref{contradiction:auchmuty} in turn implies that the sequence $\{\bar{\lambda}_j\}_{j \in \mathbb{N}}$ is bounded, hence the contradiction. 
Therefore necessarily $\bar{\lambda}_j= \lambda_j[\varepsilon]$ for every $j \in \mathbb{N}$ and the theorem is proved.
\end{proof}

\section{Permittivity optimization}\label{sec:permop}
In this section we consider the problem of finding those densities which maximize or  minimize the symmetric functions of Maxwell eigenvalues  under the constraint:
\[
\int_\Omega|\varepsilon|_{\mathcal{F}}\,dx = \mbox{const.} 
\]
As explained in the introduction, the above constraint is natural as it is related to the normalized energy capacity.

First, we need to recall what the symmetric functions of eigenvalues are. 
Given a finite set of indices $F \subset \mathbb{N}$, we consider those  permittivities $\varepsilon \in \mathcal{E}$ for which Maxwell eigenvalues with indices in $F$ do not coincide with Maxwell eigenvalues with indices outside $F$. We then introduce the following sets:
\[
\mathcal{E}[F] := \set{\varepsilon \in \mathcal{E} : \lambda_j[\varepsilon] \neq \lambda_l[\varepsilon] \ \forall j \in F, l \in \mathbb{N}\setminus F}
\]
and 
\[
\Theta[F] := \set{\varepsilon \in \mathcal{E}[F] : \lambda_j[\varepsilon] \text{ have a common value } \lambda_F[\varepsilon] \text{ for all }  j \in F}.
\]
Let $\varepsilon \in \mathcal{E}[F]$. We recall that the elementary symmetric function of degree $s \in \{1,\dots, \abs{F}\}$ of the Maxwell eigenvalues with indices in $F$ is defined by
\[
\Lambda_{F,s}[\varepsilon] = \sum_{\substack{j_1,\dots,j_s \in F \\ j_1<\dots<j_s}} \lambda_{j_1}[\varepsilon] \cdots \lambda_{j_s}[\varepsilon].
\]
Note that symmetric functions are   natural quantities to consider when one deals with the regularity and the optimization of multiple eigenvalues with respect to a parameter, which in our case is $\varepsilon$. Indeed,  if one has a multiple eigenvalue $\lambda= \lambda_{j}[\varepsilon]=\cdots = \lambda_{j+m-1}[\varepsilon]$
   and $\varepsilon$ is slightly perturbed, $\lambda$ could split into different eigenvalues of lower multiplicity and thus the corresponding branches can  present a corner at the splitting point, and then be not differentiable. Also note that classical analytic perturbation theory in the spirit of Rellich \cite{Re37} and Nagy \cite{Na48} works only with a  single real perturbation parameter, while our perturbation is infinte dimensional.
   
    As pointed out  in  \cite{LaLa04} in an abstract setting and later in several works dealing with different parameters (see, e.g., \cite{BuLa13, BuLa15, LaZa20, LaLuMu21, LaMuTa21}), this is not the case when one considers the symmetric functions of the eigenvalues since, if the parameter enters in the problem in a sufficiently regular way, they depend real analytically.
   
Following this idea, in  \cite[Thm. 4.2]{LuZa22} it is proved that the elementary symmetric functions depends real analytically upon $\varepsilon$ and it is computed their $\varepsilon$-Fr\'echet derivative.  More precisely,   the following result holds.
\begin{theorem}\label{thm:diffeps}
Let $\Omega$ be a bounded Lipschitz domain and  $\varepsilon \in \mathcal{E}$ be such that \eqref{ineq:GF} holds. Let $F$ be a finite subset of $\mathbb{N}$ and 
$s \in \{1,\ldots,|F|\}$. Then $\mathcal{E}[F]$ is  open  in $W^{1,\infty}(\Omega)^{3 \times 3}  \cap  \mathrm{Sym}_3 (\Omega)$ and the elementary symmetric function $\Lambda_{F,s}$ depends real analytically upon $\varepsilon \in \mathcal{E}[F]$.

Moreover, if $\{F_1,\ldots, F_n\}$ is a partition 
of $F$ and $\tilde \varepsilon \in \bigcap_{k=1}^n \Theta[F]$ is 
such that for each $k =1,\ldots, n$ the Maxwell eigenvalues $\lambda_j[\tilde \varepsilon]$ assume the common 
value $\lambda_{F_k}[\tilde \varepsilon]$ for all $j \in F_k$, then the differential of the function 
$\Lambda_{F,s}$ at the point $\tilde \varepsilon$ is given by the formula
\begin{equation}\label{diff:Lambdas}
d\rvert_{\varepsilon=\tilde{\varepsilon}} \Lambda_{F,s} [\eta] = -\sum_{k=1}^nc_k  \sum_{l \in F_k} \int_\Omega \eta \tilde{E}^{(l)} \cdot  \tilde{E}^{(l)} \, dx,
\end{equation} 
for all $\eta \in  W^{1,\infty}(\Omega)^{3 \times 3} \left(\Omega\right) \cap  \mathrm{Sym}_3 (\Omega)$, where
\begin{equation}\label{ckdef}
c_k :=  \sum_{\substack{0 \leq s_1 \leq |F_1| \\ \ldots \\ 0 \leq s_n \leq |F_n| \\s_1+\ldots +s_n =s}}
\binom{\, \abs{F_k}-1}{s_k-1} (\lambda_{F_k}[\tilde{\varepsilon}])^{s_k} \prod_{\substack{j=1\\j \neq k}}^n
\binom{\, \abs{F_j}}{s_j}(\lambda_{F_j}[\tilde \varepsilon])^{s_j},
\end{equation}
and for each $k=1, \dots, n$, $\{\tilde E^{(l)}\}_{l \in F_k}$ is an orthonormal basis in $L_{\tilde \varepsilon}^2(\Omega)$ of Maxwell eigenvectors for the eigenspace associated with $\lambda_{F_k}[\tilde \varepsilon]$.
\end{theorem}

We are ready to turn our attention to the optimization problems in \eqref{extremumenergysym}, which for the sake of clarity we recall  were:
\begin{equation*}
\min_{\int_{\Omega}|\varepsilon|_{\mathcal{F}}  dx={\rm const}}\Lambda_{F,s}[\varepsilon],\ \ {\rm and}\ \   \max_{\int_{\Omega}|\varepsilon_{\mathcal{F}} |dx={\rm const}}\Lambda_{F,s}[\varepsilon].
\end{equation*}
Thus,  for any $ m>0$
we set 
\[
L_m := \left\{\varepsilon \in \mathcal{E} : \int_{\Omega}|\varepsilon|_{\mathcal{F}} \,dx = m \right\}.
\]
Recall that, given a $3 \times 3$ matrix $M$, by $|M|_{\mathcal{F}}$ we denote the Frobenius norm of $M$, that is
\begin{equation*}
|M|_{\mathcal{F}} = \sqrt{M:M} = \sqrt{\mathrm{tr}(M^\top M)}= \sqrt{\sum_{i,j=1}^3 M_{ij}^2}.
\end{equation*}
Here $A:B$ denotes the Frobenius product between the two matrices $A$ and $B$.
As already mentioned, we consider the problem of finding the optimal permittivity $\varepsilon$ for the symmetric functions of the Maxwell eigenvalues under the constraint $\varepsilon \in L_m$.

In the following Theorem \ref{thm:conopt}
we show that the symmetric functions of Maxwell eigenvalues do not admit points of local extremum under the constraint 
$\varepsilon \in L_m$.  In this sense this result provides a kind of ``Maximum Principle''.
Note that Theorem~\ref{thm:conopt} is in the same spirit of the analog optimization problem for 
the eigenvalues of the Dirichlet Laplacian with mass constraint, that is
for a vibrating membrane with a fixed total mass
(see Lamberti \cite{La09}).  We exclude from our analysis the null eigenvalue since, as pointed out in the introduction, it depends only on the number of connected components of $\partial\Omega$, and thus in particular it does not depend on $\varepsilon$. To do this from now on we denote by $\kappa$ is the number of connected components of $\partial \Omega$ and we will require that the indeces in $F$ in the definition of 
$\Lambda_{F,s}$ are greater than $\kappa-1$.

\begin{theorem}\label{thm:conopt}
Let $\Omega$ be a bounded Lipschitz domain and  $\varepsilon \in \mathcal{E}$ be such that \eqref{ineq:GF} holds. Let $F$ be a finite subset of $\mathbb{N} \setminus \{1,\ldots,\kappa-1\}$. Let $s \in \{1,\ldots,|F|\}$. Let 
 $m>0$.
 Then 
the map from  $\mathcal{E}[F] \cap L_m$ to $\mathbb{R}$ which takes $\varepsilon$ to 
$\Lambda_{F,s}[\varepsilon]$ has no points of local minimum or maximum. 
\end{theorem}

\begin{proof}
Let $V$ be the map from $\mathcal{E}$ to $(0,+\infty)$  defined by 
\[
V[\varepsilon] :=  \int_{\Omega} |\varepsilon|_{\mathcal{F}} \,dx \qquad \forall \varepsilon \in 
\mathcal{E}.
\]
We proceed by contradiction. Assume that there exists a local minimum or maximum $\tilde \varepsilon \in \mathcal{E}$.  Clearly, there exists $n \in \mathbb{N}$ and a partition $\{F_1,\ldots, F_n\}$ of $F$ 
such that $\tilde \varepsilon \in \bigcap_{k=1}^n \Theta[F_k]$.  Since $\tilde \varepsilon$ is a local extremum, it is a critical point for the function $\Lambda_{F,s}[\varepsilon]$ subject to the constraint
$
V[\varepsilon] = m.
$
This implies the existence of a Langrange multiplier, which means that there exists $A \in \mathbb{R}$
such that 
\[
d\rvert_{\varepsilon=\tilde{\varepsilon}} \Lambda_{F,s} = 
A \, d\rvert_{\varepsilon=\tilde{\varepsilon}} V. 
\]
(see, e.g., Deimling \cite[Thm. 26.1]{De85}). 
By standard calculus it is not difficult to see that 
\begin{equation*}
d\rvert_{\varepsilon=\tilde{\varepsilon}} V [\eta] = \int_\Omega \frac{\tilde{\varepsilon}}{|\tilde{\varepsilon}|_{\mathcal{F}}} : \eta \, dx \qquad  \forall\eta \in W^{1,\infty}(\Omega)^{3 \times 3} \left(\Omega\right) \cap  \mathrm{Sym}_3 (\mathbb{R}).
\end{equation*}
Hence, by formula \eqref{diff:Lambdas} we have that 
\[
\sum_{k=1}^nc_k  \sum_{l \in F_k} \int_\Omega \eta \tilde{E}^{(l)} \cdot  \tilde{E}^{(l)} \, dx =A
\int_\Omega \frac{\tilde{\varepsilon}}{|\tilde{\varepsilon}|_{\mathcal{F}}} : \eta \, dx,
\]
for all $\eta \in  W^{1,\infty}(\Omega)^{3 \times 3} \left(\Omega\right) \cap  \mathrm{Sym}_3 (\mathbb{R})$,
where   for each $k=1, \dots, n$, $\{\tilde E^{(l)}\}_{l \in F_k} \subseteq H^1(\Omega)^3$ is an orthonormal basis in $L_{\tilde \varepsilon}^2(\Omega)$ of Maxwell eigenvectors for the eigenspace associated with $\lambda_{F_k}[\tilde \varepsilon]$. Recall that $c_k$, $k=1,\ldots,n$, are the constants defined in \eqref{ckdef}.

Since $\eta$ is arbitrary, by the Fundamental Lemma of Calculus of Variations we have that  
\begin{equation*}A
\tilde{\varepsilon}_{ij} = |\tilde{\varepsilon}|_{\mathcal{F}}\sum_{k=1}^n c_k  \sum_{l \in F_k}  \tilde{E}^{(l)}_i  \tilde{E}^{(l)}_j \quad \text{a.e. in }\Omega,
\end{equation*}
for all $i,j=1,\dots , 3$. Note that by the previous equality with  $i=j$, combined with   the ellipticity condition in  \eqref{def:Aabc},  we immediately deduce that $A> 0$. 
  Since $\varepsilon$ is continuous up to the boundary of $\Omega$ and the eigenfunctions belong to $H^1(\Omega)^3$, we can trace the previous equality at the boundary to deduce that 
\begin{equation}\label{traccialagrange}A
\tilde{\varepsilon}_{ij}(x) = |\tilde{\varepsilon}(x)|_{\mathcal{F}}\sum_{k=1}^n c_k  \sum_{l \in F_k}  \tilde{E}^{(l)}_i (x) \tilde{E}^{(l)}_j(x)\qquad \text{for a.e.  }x\in \partial \Omega .
\end{equation}
Note that by the boundary conditions it follows  that $E^{(l)} = B^{(l)} \nu$ on $\partial \Omega$ for some scalar functions  $B^{(l)} $. Thus equality \eqref{traccialagrange} can be written in the form
\begin{equation}\label{traccialagrange1}A
\tilde{\varepsilon}_{ij}(x) = |\tilde{\varepsilon}(x)|_{\mathcal{F}}\sum_{k=1}^n c_k  \sum_{l \in F_k} |B^{(l)}(x) |^2\nu_i(x)\nu_j(x) \qquad \text{for a.e.  }x\in \partial \Omega .
\end{equation}
Consider a point $\hat x\in \partial \Omega$ such that equality \eqref{traccialagrange1} holds and fix a vector $\hat \xi \in \mathbb{R}^3\setminus\{0\}$ such that $\nu (\hat x)\cdot \hat\xi=0$. 
Since $\tilde{\varepsilon}\in \mathcal{E}$, using equality \eqref{traccialagrange1} at the point $\hat x$ and the ellipticity condition in  \eqref{def:Aabc} with $\xi =\hat \xi$, we get
\begin{equation*} \label{ineqassurdo}
0\le \alpha A|\hat \xi|^2 \le  A\tilde{\varepsilon}(\hat x)\hat\xi \cdot  \hat \xi = |\tilde{\varepsilon}(\hat x)|_{\mathcal{F}}\sum_{k=1}^n c_k  \sum_{l \in F_k} |B^{(l)}(\hat x) |^2( \nu (\hat x)\cdot \hat\xi )^2=0
\end{equation*}   
which implies that $A=0$, a contradiction. 
\end{proof}

Finally, the previous maximum principle can be used to deduce a result on the localization of optimal permittivities. 
\begin{corol}
 Let $\alpha,\beta,\gamma>0$  with $0 < \alpha < \beta$. Let $\Omega$ be as in \eqref{Omega_def}. Let $F$ be a finite subset of $\mathbb{N} \setminus \{1,\ldots,\kappa-1\}$. Let $s \in \{1,\ldots,|F|\}$. Let 
 $m>0$. Let $C\subseteq \mathcal{A}_{\alpha,\beta,\gamma}$ be a weakly* compact set in $W^{1,\infty}(\Omega)^{3 \times 3}$. Then  the map $\Lambda_{F,s}$  
 from $C \cap L_m$ to $\mathbb{R}$ admits points of minimum and maximum and all this points are in
 $\partial C \cap  L_m$.
\end{corol}

\begin{proof} Note that $C$ is bounded in $W^{1,\infty}(\Omega)^{3 \times 3}$ being  a subset of $\mathcal{A}_{\alpha,\beta,\gamma}$. This, combined with Remark~\ref{strLinf}, allows to prove that $C\cap  L_m$  is sequentially weakly* closed. It follows that  $C\cap  L_m$ is also weakly* compact.   
Since by Theorem \ref{weak*:cont:thm} the map $\Lambda_{F,s}$ is weakly* continuous in $ \mathcal{A}_{\alpha,\beta,\gamma}$ and $C\subseteq \mathcal{A}_{\alpha,\beta,\gamma}$ is  weakly* compact,
then it admits both a maximum and a minimum on the weakly* compact subset $C\cap  L_m$ of $C$. 
Since by Theorem \ref{thm:conopt} the corresponding maximum and minimum points can not be interior points of $C$, then they belong to $\partial C \cap  L_m$.
\end{proof}

\appendix

\section{Auchmuty principle}\label{sec:auc}

The following part is an adaptation to problem \eqref{prob:eigen2weak} of the arguments from 
Auchmuty \cite{auc}. It is worth noting that for these results we would only require $\xn^\varepsilon(\Omega)$ to be compactly embedded in $L^2_\varepsilon(\Omega)$, with possibly less assumptions on the permittivity parameter $\varepsilon$ and on the set $\Omega$.
Nonetheless, we will assume $\varepsilon$ to be in the admissible class $\mathcal{E}$ and  $\Omega$ to be a bounded domain of $\mathbb{R}^3$ as in \eqref{Omega_def}.

We define the functional $\hat{f} : \xn^\varepsilon(\Omega) \to \mathbb{R}$ as follows:
\begin{equation*} \label{fhat:definit}
\hat{f}(u):= \frac{1}{2} T_\varepsilon [u,u] \qquad \forall\, u \in \xn^\varepsilon(\Omega),
\end{equation*}
where we recall that the bilinear form $T_\varepsilon$ has been defined in \eqref{def:Te}.
Let $M \in \mathbb{N}$ and suppose $\sigma_1[\varepsilon], \dots, \sigma_M[\varepsilon]$ are the first $M$ eigenvalues of \eqref{prob:eigen2weak} repeated according to their multiplicity and $u_1,\dots, u_M$ a corresponding $L^2_\varepsilon(\Omega)$-orthonormal set of eigenfunctions. Let $P_M$ be the $L^2_\varepsilon(\Omega)$-orthogonal projection  of $\xn^\varepsilon(\Omega)$ onto the subspace spanned by $\{ u_1, \dots, u_M \}$, that is 
\begin{equation*}
P_M u = \sum_{i=1}^M \langle u, u_i \rangle_\varepsilon u_i \quad  \forall\, u \in \xn^\varepsilon(\Omega).
\end{equation*}
Then, we define $f_{M+1}: \xn^\varepsilon(\Omega) \to \mathbb{R}$ as follows
\begin{equation} \label{def:fM+1}
f_{M+1}(u) := \hat{f}(u) - \norm{(I-P_M) u}_{L^2_\varepsilon(\Omega)} \qquad \forall\, u \in \xn^\varepsilon(\Omega).
\end{equation}

\begin{rem}
If $\tilde{u} \in \xn^\varepsilon(\Omega)$ is an eigenfunction of \eqref{prob:eigen2weak} with eigenvalue $\tilde{\sigma}[\varepsilon]$, then clearly by \eqref{prob:eigen3weak} one has that
\begin{equation} \label{eq:auchm:luzzini}
f_{M+1}(\tilde{u}) = \frac{1+ \tilde{\sigma}[\varepsilon]}{2} \norm{\tilde{u}}_{L^2_\varepsilon(\Omega)}^2 - \norm{(I-P_M) \tilde{u}}_{L^2_\varepsilon(\Omega)}.
\end{equation}
\end{rem}
We begin with proving some properties of the functional $f_{M+1}$.
\begin{lem} \label{lemma:weaklsc:M+1}
The functional $f_{M+1}$ is sequentially weakly lower semi-continuous and coercive on $\xn^\varepsilon(\Omega)$. Moreover, it is G\^ateaux-differentiable on $\xn^\varepsilon(\Omega) \setminus \{0\}$ and the G\^ateaux-differential
$df_{M+1}(u;v)$ of $f_{M+1}$ at the point $u \in \xn^\varepsilon(\Omega) \setminus \{0\}$ and in the direction of $v \in  \xn^\varepsilon(\Omega)$ is given by
\begin{equation} \label{gateaux:der:M+1}
 df_{M+1}(u;v) = T_\varepsilon [u,v] - \norm{(I-P_M)u}_{L^2_\varepsilon(\Omega)}^{-1}  \langle(I-P_M)u, (I-P_M)v\rangle_\varepsilon
\end{equation}
for all $v \in \xn^\varepsilon(\Omega)$.
\end{lem}
\begin{proof} 
Let $u \in \xn^\varepsilon(\Omega) \setminus \{0\}$.
Since 
\[
\norm{(I-P_M)u}_{L^2_\varepsilon(\Omega)} \leq \norm{u}_{L^2_\varepsilon(\Omega)} \leq \norm{u}_{\xn^\varepsilon(\Omega)},
\]
 then
$$\frac{f_{M+1}(u)}{\norm{u}_{\xn^\varepsilon(\Omega)}} \geq \frac{1}{2}\,\frac{T_\varepsilon[u,u]}{\norm{u}_{\xn^\varepsilon(\Omega)}} -1 \to +\infty \quad \text{as } \norm{u}_{\xn^\varepsilon(\Omega)} \to +\infty,$$
which shows that is  $f_{M+1}$ is coercive.

It is not difficult to see that the functional $\hat{f}= \frac{1}{2}\, T_\varepsilon$ is strictly convex, that is for any $u,v \in \xn^\varepsilon(\Omega), u \neq v$, we have that 
\begin{equation*}
\hat{f} \left(\frac{u+v}{2} \right) < \frac{\hat{f}(u) + \hat{f}(v)}{2}.
\end{equation*}
Moreover, since $\hat{f}$ is strongly continuous on $\xn^\varepsilon(\Omega)$ and strictly convex,
it is also weakly lower semi-continuous.
Additionally, the embedding $\iota_\varepsilon :\xn^\varepsilon(\Omega) \to L^2_\varepsilon(\Omega)$ is compact, therefore the $L^2_\varepsilon$-norm is sequentially weakly continuous on $\xn^\varepsilon(\Omega)$.
Since the orthogonal projection $P_M$ is also weakly continuous, 
we have that $f_{M+1}$ is  sequentially weakly lower semi-continuous.

Finally, the G\^ateaux derivative is just a simple computation and therefore we omit the details.
\end{proof}

In the sequel we shall  also need the following.
\begin{corol} \label{coroll:critpoint:M+1}
Let $\hat{u} \in \langle u_1, \dots, u_M \rangle_{L^2_\varepsilon(\Omega)}^\perp$ be in the  $L^2_\varepsilon(\Omega)$-orthogonal space  of the span of $\{u_1, \dots, u_M\}$. Suppose moreover that $\hat{u}$ is a non-zero critical point of $f_{M+1}$ on $\xn^\varepsilon(\Omega)$. Then $\hat{u}$ solves problem \eqref{prob:eigen2weak} with associated eigenvalue $\sigma[\varepsilon]= \norm{\hat{u}}_{L^2_\varepsilon(\Omega)}^{-1}-1$.
\end{corol}
\begin{proof}
If $0 \neq \hat{u} \in \xn^\varepsilon(\Omega)$ is a critical point of $f_{M+1}$ then $df_{M+1}(\hat{u};v) =0$ for all $v \in \xn^\varepsilon(\Omega)$. By formula \eqref{gateaux:der:M+1} and the fact that $\hat{u}$ is $L^2_\varepsilon$-orthogonal to $u_i$ for all  $i=1,\dots,M$, which in turns implies that $P_M \hat{u}=0$ and $\langle\hat{u}, P_M v\rangle_\varepsilon=0$ for any $v \in \xn^\varepsilon(\Omega)$, we see that $\hat{u}$ solves problem \eqref{prob:eigen2weak} with  $\sigma[\varepsilon]= \norm{\hat{u}}^{-1}_{L^2_\varepsilon(\Omega)} -1$.
\end{proof}

We can finally prove the following theorem. The proof follows the lines of Auchmuty  \cite[Thm. 7.4]{auc} where the author proves the corresponding result in the case of an elliptic operator.
\begin{theorem}[Auchmuty Principle] \label{thm:min:f:M+1}
The functional $f_{M+1}$ defined in \eqref{def:fM+1} attains its minimum $-\frac{1}{2}(\sigma_{M+1}[\varepsilon] +1)^{-1}$ on $\xn^\varepsilon(\Omega)$. This minimum is attained at $\pm \hat u_{M+1}$ where $\hat u_{M+1}$ is a solution of \eqref{prob:eigen2weak} corresponding to the eigenvalue $\sigma_{M+1}[\varepsilon]$ and such that $\norm{\hat u_{M+1}}_{L^2_\varepsilon(\Omega)} = (\sigma_{M+1}[\varepsilon] +1)^{-1}$.
\end{theorem}
\begin{proof}
Since $\xn^\varepsilon(\Omega)$ is reflexive and $f_{M+1}$ is sequentially weakly lower semi-continuous and coercive  by Lemma \ref{lemma:weaklsc:M+1}, we have that $f_{M+1}$ attains a finite minimum on $\xn^\varepsilon(\Omega)$.
Observe that $f_{M+1}(0)=0$.
We complete the orthonormal set of eigenfunctions $u_1, \dots, u_M$ to a $L^2_\varepsilon$-orthonormal basis $\{u_i\}_{i \in \mathbb{N}} \subset \xn^\varepsilon(\Omega)$.
Now, for any $u \in \xn^\varepsilon(\Omega)$ write $u = P_M u + w$, where 
\[
P_M u = \sum_{i=1}^M \langle u,u_i\rangle_\varepsilon\,  u_i\,\, , \qquad \quad w=(I-P_M)u= \sum_{i=M+1}^{\infty} \langle u,u_i\rangle_\varepsilon\, u_i\,.
\]
Observe that $w \in \langle u_1, \dots, u_M \rangle_{L^2_\varepsilon(\Omega)}^\perp$, thus the vector fields $P_M u$ and $w$ are orthogonal in $L^2_\varepsilon(\Omega)$, that is  $\langle P_M u, w\rangle_\varepsilon=0$. 
Moreover, observe that for all $i=1, \dots, M$ one has
\begin{equation} \label{u_i:eigenfunction}
\int_\Omega \operatorname{curl}u_i \cdot \operatorname{curl}v \, dx + \int_\Omega \operatorname{div}(\varepsilon u_i) \operatorname{div}(\varepsilon v) \, dx = \sigma_i[\varepsilon] \int_{\Omega} \varepsilon u_i \cdot v\,dx \qquad 
\forall \,v \in \xn^\varepsilon(\Omega).
\end{equation}
Since   $\langle u_i, w\rangle_\varepsilon=0$ for all $i=1, \dots, M$  then, setting $v=w$ in \eqref{u_i:eigenfunction}, we see that 
\[
\int_\Omega \operatorname{curl}u_i \cdot \operatorname{curl}w \, dx + \int_\Omega \operatorname{div}(\varepsilon u_i) \operatorname{div}(\varepsilon w) \, dx =0.
\]
 Hence 
\begin{equation*}
\int_\Omega \operatorname{curl}P_M u \cdot \operatorname{curl}w \, dx + \int_\Omega \operatorname{div}(\varepsilon P_M u) \operatorname{div}(\varepsilon w) \, dx =0.
\end{equation*}
Therefore $T_\varepsilon [u][u] = T_\varepsilon [P_M u,P_M u] + T_\varepsilon [w,w]$ and
\begin{equation*}
f_{M+1}(u) = \frac{1}{2} T_\varepsilon[u][u] - \norm{w}_{L^2_\varepsilon(\Omega)} \geq \frac{1}{2} T_\varepsilon [w,w] -\norm{w}_{L^2_\varepsilon(\Omega)} = f_{M+1}(w).
\end{equation*}
Thus the minimum of $f_{M+1}$ occurs in the orthogonal space $\langle u_1, \dots, u_M \rangle_{L^2_\varepsilon(\Omega)}^\perp$. If $\hat{u} \in \langle u_1, \dots, u_M \rangle_{L^2_\varepsilon(\Omega)}^\perp$ is a non-zero critical point of $f_{M+1}$, by Corollary \ref{coroll:critpoint:M+1} one has that $\hat{u}$ is an eigenfunction of problem \eqref{prob:eigen2weak} associated with the eigenvalue $\sigma[\varepsilon]= \norm{\hat{u}}^{-1}_{L^2_\varepsilon(\Omega)} -1$. Hence, observing that $P_M \hat{u}=0$ and using \eqref{eq:auchm:luzzini} we get that
\begin{equation*}
f_{M+1}(\hat{u}) = -\frac{1}{2} \norm{\hat{u}}_{L^2_\varepsilon(\Omega)} = -\frac{1}{2} (\sigma[\varepsilon] +1)^{-1} <0.
\end{equation*}
In particular $f_{M+1}$ is minimized when $\sigma[\varepsilon]$ is the smallest eigenvalue of problem \eqref{prob:eigen2weak} corresponding to an eigenfunction $\hat{u}$ which is orthogonal to $u_1, \dots, u_M$, that is when $\sigma[\varepsilon]= \sigma_{M+1}[\varepsilon]$  and such a function $\hat u$  can be chosen as the function $\hat{u}_{M+1}$  required in the statement. 
\end{proof}

\section*{Acknowledgments}

  The authors are very thankful to  Prof. A. Lakhtakia and to Prof. Ioannis G. Stratis for useful discussions  during the preparation of the paper, in particular with respect to the physical motivation of the problem as presented in the introduction.
P.D.L and P.L. are members of the ``Gruppo Nazionale per l’Analisi Matematica, la Probabilit\`a e le loro Applicazioni'' (GNAMPA) of the ``Istituto Nazionale di Alta Matematica'' (INdAM) and acknowledge the support 
  of the ``INdAM GNAMPA Project'' codice CUP\_E53C22001930001 ``Operatori differenziali e integrali in geometria spettrale''.  M.Z.   acknowledges  financial support by INdAM   through their program ``Borse di studio per l'estero''. P.D.L. and P.L. also acknowledge the support  from the project ``Perturbation problems and asymptotics for elliptic differential equations: variational and potential theoretic methods'' funded by the MUR Progetti di Ricerca di Rilevante Interesse Nazionale (PRIN) Bando 2022 grant 2022SENJZ3.

\end{document}